\documentclass[12pt]{amsart}

\usepackage{amsmath, amsthm, amssymb}
\usepackage{graphicx}
\usepackage{mathbbol}
\usepackage{txfonts}
\usepackage[usenames]{color}
\usepackage{pgfplots}
\usepackage{comment}
\usepackage{enumerate}
\usepackage{tikz}
\usetikzlibrary{calc}

\newtheorem{thm}{Theorem}[section]
\newcommand{\bt}{\begin{thm}}
\newcommand{\et}{\end{thm}}

\newtheorem{cor}[thm]{Corollary}   
\newcommand{\bc}{\begin{cor}}
\newcommand{\ec}{\end{cor}}

\newtheorem{lem}[thm]{Lemma}   
\newcommand{\bl}{\begin{lem}}
\newcommand{\el}{\end{lem}}

\newtheorem{prop}[thm]{Proposition}
\newcommand{\bp}{\begin{prop}}
\newcommand{\ep}{\end{prop}}

\newtheorem{defn}[thm]{Definition}
\newcommand{\bd}{\begin{defn}}    
\newcommand{\ed}{\end{defn}}

\newtheorem{rmrk}[thm]{Remark}   

\newcommand{\br}{\begin{rmrk}}
\newcommand{\er}{\end{rmrk}}

\newtheorem{example}[thm]{Example}

\newtheorem{question}[thm]{Question}

\newcommand{\mina}{\operatorname{MinA}}

\newcommand{\Scal}{\operatorname{Scalar}}

\newcommand{\be}{\begin{equation}}

 \newcommand{\ee}{\end{equation}}

\newcommand{\R}{\mathbb{R}}

\newcommand{\diam}{\operatorname{Diam}}

\newcommand{\vol}{\operatorname{Vol}}

\newcommand{\Sph}{{\mathbb S}}         

\newcommand{\scal}{\text{Scal}}

\pgfplotsset{compat=1.18}
\begin{document}

\title[More Extreme Limits]{More Extreme Limits of Manifolds with Positive Scalar Curvature}

\author{Wenchuan Tian}
\address[Wenchuan Tian]{University of California, Santa Barbara}
\email{tian.wenchuan@gmail.com}

\date{}

\begin{abstract}
In this article, we extend the example constructed in \cite{STW} to build new examples that satisfy the assumptions of the conjecture by Gromov. Each of these new examples of sequence converges to a limit space with infinitely many poles in $\Sph^2$. These examples can be used to test various notions of weak scalar curvature.
\end{abstract}

\maketitle

\section{\bf Introduction}

In \cite{Gromov-Dirac} and \cite{Gromov-Plateau}, Gromov conjectured that if a sequence of Riemannian manifolds have a non-negative scalar curvature and are subject to a uniform volume upper bound and a uniform diameter upper bound, then there exists a subsequence that converges in some weak sense to a limit space. Moreover, this limit space would have a generalized notion of ``non-negative scalar curvature".

Several works explored various possibilities for how to characterize non-negative generalized scalar curvature. 

Firstly, in \cite{B-G-GAFA}, Burkhardt-Guim uses Ricci flow to define the notion of a scalar curvature lower bound in the $\beta-$weak sense. See also \cite{B-G-SIGMA-survey}. Burkhardt-Guim's work extends earlier work by Gromov \cite{Gromov-Dirac} and Bamler \cite{Bamler-Gromov}. Huang and Lee \cite{Huang-Lee-scalar} generalize Bamler's proof in another direction. 

Secondly, in \cite{Lee-LeFloch}, Lee and LeFloch defined the notion of scalar curvature in the sense of distribution for Riemannian metric tensors that are locally $L^\infty \cap W^{1,2}$ and have locally $L^\infty$ inverse. See also the work of LeFloch and Mardare \cite{LM07}. In \cite{JSZ23}, Jiang, Sheng, and Zhang proved that the distributional scalar curvature lower bound is preserved along Ricci flow when starting from a $W^{1,p}$ Riemannian metric tensor on an n-dimensional compact manifold for $p\in (n,\infty]$.

Thirdly, we can use the volume-limit notion to characterize scalar curvature or scalar curvature lower bounds. This notion is used by Basilio, Dodziuk, and Sormani \cite{BDS-Sewing}, Basilio and Sormani \cite{BS-seq}, and Kazaras and Xu \cite{KX-Drawstring} to show that the limit space has negative scalar curvature at some point.

Lastly, Li \cite{Li-poly} proved a comparison theorem for polyhedra in $3-$dimensional Riemannian manifolds with positive scalar curvature. We can use Riemannian Polyhedral comparison to characterize lower bounds on scalar curvature in closed $3$-dimensional Riemannian manifolds with positive scalar curvature.

Previously, the author worked with several collaborators to explore Gromov's conjecture with the additional $\mina$ lower bound. In \cite{Park-Tian-Wang-18}, the author worked with Jiewon Park and Changliang Wang to confirm Gromov's conjecture under the additional $\mina$ lower bound for sequences of rotationally symmetric Riemannian manifolds. In \cite{Tian-Wang}, the author worked with Changliang Wang to explore Gromov's conjecture under the additional $\mina$ lower bound for sequences of warped product $\Sph^2\times \Sph^1$, we prove that the sequence of Riemannian metric converges to a limit Riemannian metric that is $W^{1,p}$ for all $p\in [1,2)$. We also prove that the limit space has non-negative distributional scalar curvature as defined in \cite{Lee-LeFloch}.

In \cite{STW}, the author collaborated with Christina Sormani and Changliang Wang to construct a sequence of warped product $\Sph^2\times \Sph^1$ that satisfies the assumptions in Gromov's conjecture and also satisfies the uniform $\mina$ lower bound. When the limit is taken, the resulting space is an extreme limit space with two poles located at the north and south poles of the standard $\Sph^2$. This example shows that the regularity result in Theorem 1.3 and Theorem 1.7 of \cite{Tian-Wang} is sharp.

In this paper, we construct three different sequences of warped product $\Sph^2\times  \Sph^1$ in Example \ref{Example-Case1}, \ref{Example-Case2}, and \ref{Example-Case3}. These examples of sequences satisfy the assumptions in Gromov's conjecture, and each of these sequences converges to an extreme limit space with infinitely many poles in $\Sph^2$. The sequence in Example \ref{Example-Case2} even converges to a limit space with infinitely many poles that are dense in $\Sph^2$. 

We start with the definition of the following function
\begin{defn}\label{Defn-functions}
Let $a>0$ and $b\geq 2$. Define $f_{a,b}:[0,\pi]\to [2,\infty)$ by
\be
f_{a,b}(r)= \ln\left(\frac{1+a}{\sin^2 r+a}\right)+b. 
\ee
\end{defn}

Previously, in \cite{STW}, we used this function to construct an example of a sequence. We replace the variable $r$ with the standard $\Sph^2$ distance from a pole in that example. We proved that this function is smooth on $\Sph^2$ and defines a smooth Riemannian metric tensor that is a warped product metric in $\Sph^2\times \Sph^1$. The function $f_{a,b}(r)$ increases monotonically as $a$ decreases, and it converges to $f_{\infty}(r)=-2\ln(\sin r)+b$ as $a\to 0$. Hence the sequence of warped product metric converges to an extreme limit space that has two $\Sph^1$ fibers that stretch to infinite at the two poles in the standard $\Sph^2$.

In this paper, we construct three examples in which fibers stretch to infinite at a collection of points that are countably infinite. These examples are special cases of a general class of example in Definition \ref{General-Case}. To do so, we choose a countable collection of points $\{x_{ij}\}_{i,j=1}^\infty$ in $\Sph^2$ and sum over functions as in Definition \ref{Defn-functions} where for each function, we replace $r$ by $r_{ij}$, the $\Sph^2$ distance to $x_{ij}$.

In the following, we define a class of examples using the function in Definition \ref{Defn-functions}.
\begin{defn}\label{Sub-General-Case}
For $j=1,\ 2,...$, choose $a_j>0$ such that $a_{j+1}\leq a_{j}$ and that $a_j\to 0$ as $j\to \infty$. Let $\bar{K}\in [2,\infty)$, choose a sequence $\{b_i\}_{i=1}^\infty$ such that $b_i\in [2,\bar{K}]$ for all $i$. Let $K>0$, choose a sequence $\{A_i\}_{i=1}^\infty$ such that $A_i\geq 0$ for all $i$, $A_1>0$, and that $\sum_{i=1}^\infty A_{i}  =K <\infty$. 

Let $\{x_{ij}\}_{i,j=1}^\infty$ be a countable collection of points in $\Sph^2$. Let 
\be
r_{ij}: \Sph^2\to [0,\pi] \textrm{ be }r_{ij}(x)=d_{\Sph^2}(x,x_{ij}).
\ee
For each $j$, let $\eta_j:\Sph^2\to \R^+$ be
\be
\eta_j(x)=\sum_{i=1}^j A_{i} f_{a_j,b_i}(r_{ij}(x)).
\ee 
Where $ f_{a_j,b_i}$ is as in Definition \ref{Defn-functions}.

Define a sequence of warped product metrics on $\Sph^2\times \Sph^1$ by
\be
g_j= g_{\Sph^2}+\eta_j^2 g_{\Sph^1}.
\ee
We use the notation $\Sph^2\times_{\eta_j}\Sph^1$ to denote $(\Sph^2\times\Sph^1, g_j)$.
\end{defn}

\begin{thm}\label{MT-New}
For the sequence $\Sph^2\times_{\eta_j}\Sph^1$ defined in Definition \ref{General-Case}, let $\scal_j$ denote the scalar curvature of each $\Sph^2\times_{\eta_j}\Sph^1$, then we have
\be
\scal_j\geq 0 \text{ for all }j.
\ee
Moreover, if we define $v=16\pi^2 A_1$, $V=2\pi K\left( 8\pi-4\pi\ln 4+4\pi \bar{K}\right)$,  and $D=4\pi+2\pi K\left( 8\pi-4\pi\ln 4+4\pi \bar{K}\right)$ then we have
\be
v\leq \vol(\Sph^2\times_{\eta_j}\Sph^1)\leq  V \text{ and }\diam(\Sph^2\times_{\eta_j}\Sph^1) \leq D \text{ for all }j.
\ee
After possibly passing to a subsequence, the sequence of Riemannian metric tensor $g_j$ converges in $L^q$ norm to a limit metric $g_\infty$ for all $q\in [1,\infty)$. The limit metric $g_\infty$ is in $W^{1,p}$ for all $[1,2)$. The Riemannian manifold $(\Sph^2\times \Sph^1, g_\infty)$ has a $W^{1,p}$ metric tensor with non-negative distributional scalar curvature as defined in \cite{Lee-LeFloch}.
\end{thm}
This theorem is proved at the end of Section \ref{Sec-Examples} as a consequence of Proposition \ref{Prop-Inclusion}, Theorem \ref{MT}, and Corollary \ref{cor-conv}. Theorem \ref{MT} and Corollary \ref{cor-conv} applies to a larger class of examples to be defined in Definition \ref{General-Case}.

Next, we construct the three examples. We choose three different collections of points in $\Sph^2$ and define three different examples
\begin{example}[Case 1, a sequence of converging points, see Figure \ref{f-case1}]\label{Example-Case1}
Let $\{\bar{x}_i\}_{i=1}^\infty$ be a sequence of points in $\Sph^2$ such that $\bar{x}_{i}$ converges to some $ \bar{x}_\infty $ in $\Sph^2$. Let $x_{ij}=\bar{x}_{i}$ for all $i,\ j$. In this way, we define a sequence of warped product metrics in $\Sph^2\times \Sph^1$ as in Definition \ref{General-Case}. Note that this example does not depend on $j$.
\end{example}

\begin{figure}[h]
\centering
\includegraphics[width=13cm]{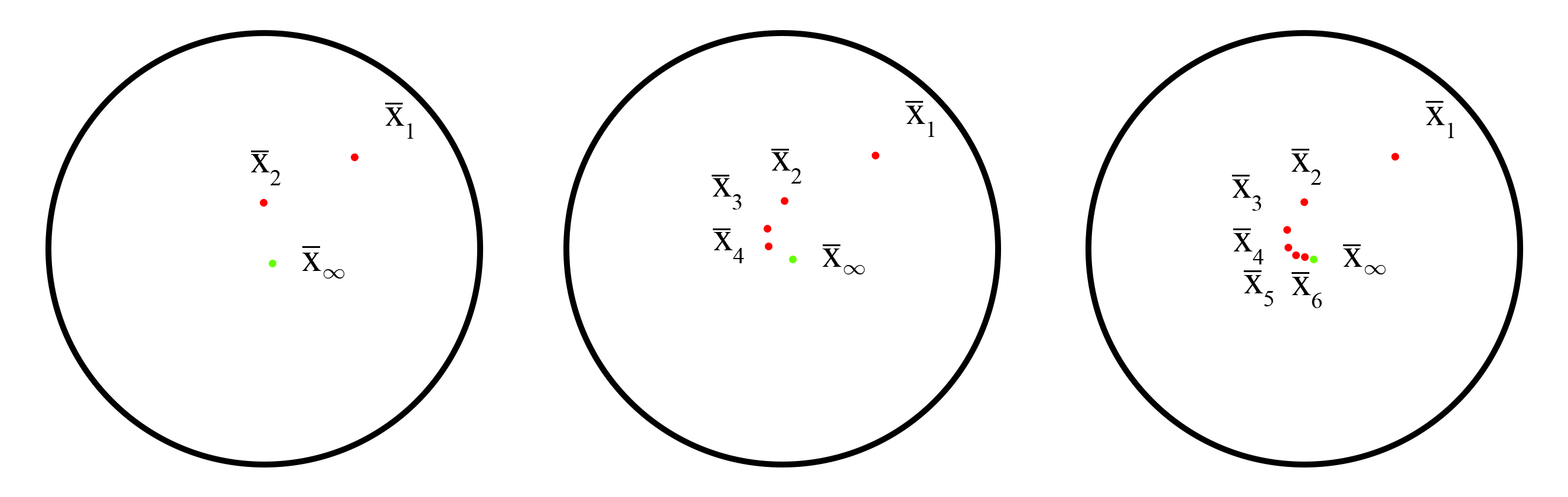}
\caption{The graph of Example \ref{Example-Case1}, a seuqnce of points in the standard $\Sph^2$ converging to $\bar{x}_\infty$. The graph on the left displays the first two points in the sequence $\{\bar{x}_i \}_{i=1}^\infty$ in red color. The graph in the middle displays the first four points in the sequence $\{\bar{x}_i \}_{i=1}^\infty$ in red color. The graph on the right displays the first 6 points in the sequence $\{\bar{x}_i \}_{i=1}^\infty$ in red color. The three graphs include the limit point $\bar{x}_\infty$ as a green dot.}
\label{f-case1}
\end{figure}

\begin{example}[Case 2, dense points in $\Sph^2$, see Figure \ref{f-case2}]\label{Example-Case2}
Choose a collection of countable dense points in $\Sph^2$. For example, we can consider a fixed polar coordinate $(r,\theta)$ in $\Sph^2$ and choose all the points with rational $r$ and $\theta$ value. Order this countable dense collection of points, and denote them as $\{\hat{x}_i\}_{i=1}^\infty$. Let $x_{ij}=\hat{x}_{i}$ for all $i,\ j$. In this way, we define a sequence of warped product metrics in $\Sph^2\times \Sph^1$ as in Definition \ref{General-Case}. Note that this example does not depend on $j$.
\end{example}

\begin{figure}[h]
\centering
\includegraphics[width=13cm]{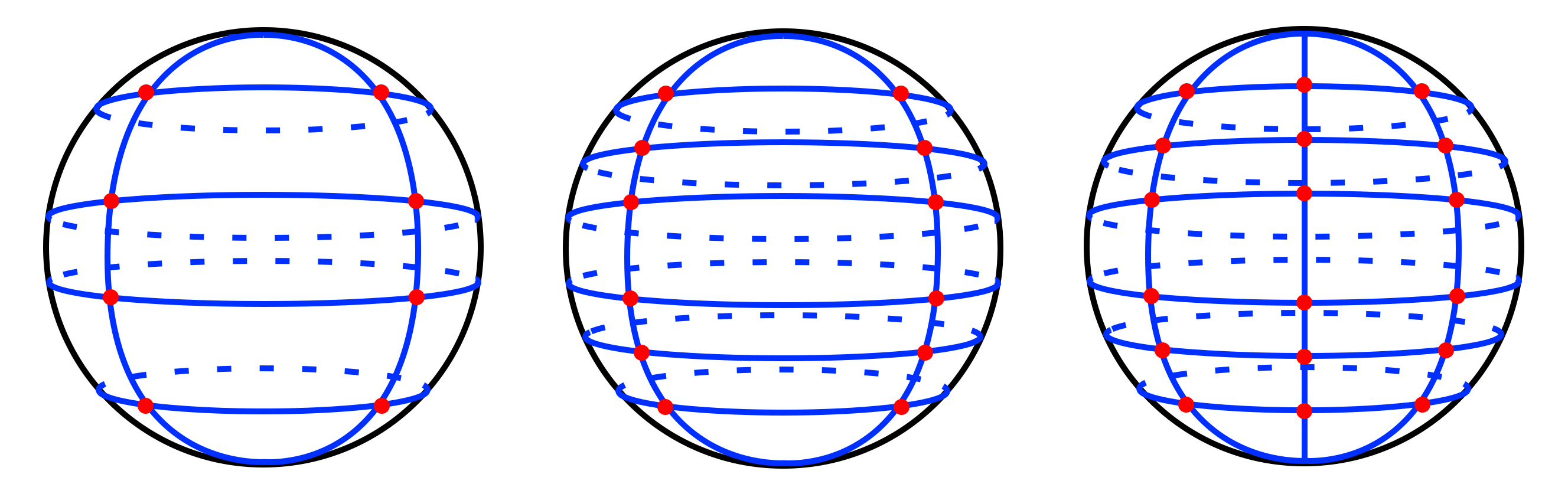}
\caption{The graph of Example \ref{Example-Case2}, a desnse collection of points in the standard $\Sph^2$. The blue line shows the longitude $\theta$ and latitude $r$ within a fixed polar coordinate. We show the chosen points in the front as red dots. The chosen points in the back are not shown in this graph. From left to right shows an increasing number of poles chosen at rational values of $r$ and $\theta$ in the fixed polar coordinate.}
\label{f-case2}
\end{figure}

\begin{example}[Case 3, moving points on the equator of $\Sph^2$, see Figure \ref{f-case3}]\label{Example-Case3}
Let $(\Sph^2, g_{\Sph^2})$ be the two-dimensional sphere with the standard metric such that $g_{\Sph^2}$ defines the distance function $d_{\Sph^2}$ in $\Sph^2$. Consider a fixed polar coordinate $(r,\theta)$ in $\Sph^2$ and choose the points on the equator such that for $i,\ j=1,\ 2,...$
\be
\tilde{x}_{ij}=(\pi/2, 2\pi i/j).
\ee
Let $x_{ij}=\tilde{x}_{ij}$ for all $i,\ j$. In this way, we define a sequence of warped product metrics in $\Sph^2\times \Sph^1$ as in Definition \ref{General-Case}.
\end{example}

\begin{figure}[h]
\centering
\includegraphics[width=13cm]{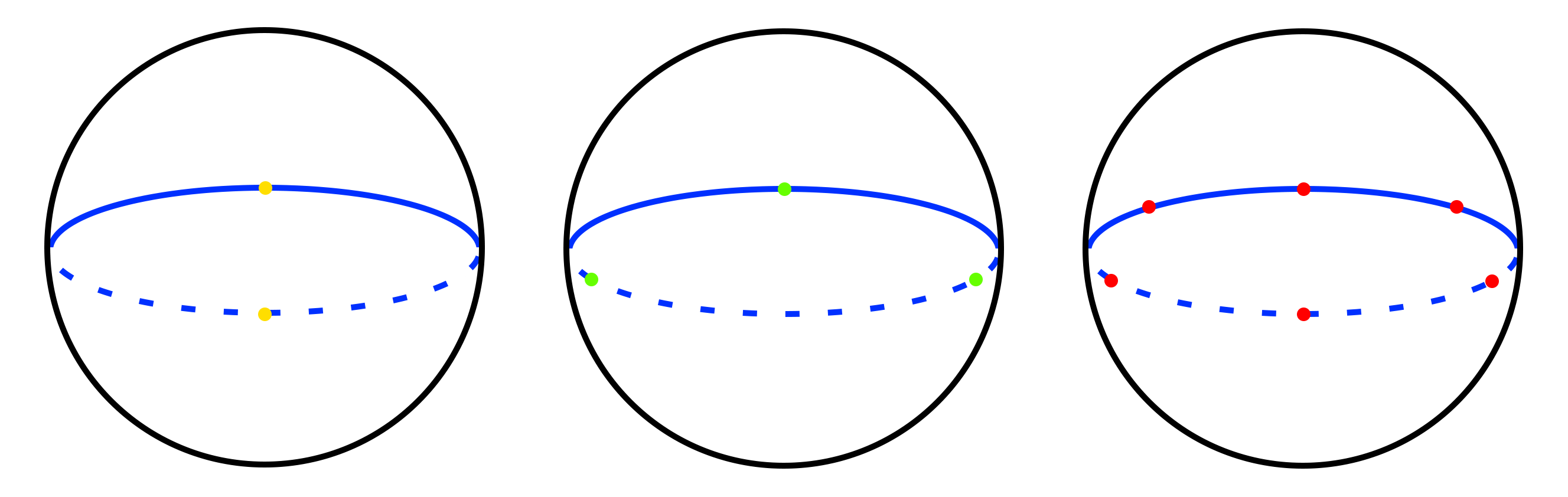}
\caption{The graph of Example \ref{Example-Case3}, a collection of moving points along the equator in $\Sph^2$. There are two chosen points in the left graph (shown in yellow), three chosen points in the middle graph (shown in green), and six chosen points in the right graph (shown in red).}
\label{f-case3}
\end{figure}

These examples can be used to test various notions of weak scalar curvature mentioned above.

This article is organized as follows. In Section \ref{Sec-Class}, we state Theorem \ref{MT}, which concerns a larger class of examples to be defined in Definition \ref{General-Case}. It asserts that the general class of examples in Definition \ref{General-Case} satisfies the assumptions in Gromov's conjecture. We explain in Remark \ref{rmk-mina} why we can apply the results from \cite{Tian-Wang} to the general class of examples as in Definition \ref{General-Case}. Then, in Corollary \ref{cor-conv}, we apply Theorem \ref{MT} and the results from \cite{Tian-Wang} to conclude the convergence of the Riemannian metric tensor and prove that the scalar curvature in the limit space is non-negative as a distribution as defined in \cite{Lee-LeFloch}. In Section \ref{Sec-Examples}, we prove in Propositions \ref{Prop-Inclusion} that the class of example in Definition \ref{Sub-General-Case} is encompassed within the general class Definition \ref{General-Case}. Then, at the end of Section \ref{Sec-Examples}, we prove Theorem \ref{MT-New}. In Section \ref{Sec-Lim}, we explicitly write down the formula for the limit metric tensor of the sequences presented in Example \ref{Example-Case1} and \ref{Example-Case2}, and prove $W^{1,p}$ convergence of the metric tensor. This is stronger than the convergence result in Corollary \ref{cor-conv}. In Section \ref{Sec-Thm}, we prove Theorem \ref{MT}. Lastly, in Section \ref{Sec-Scalar-Curvature}, we ask questions about how to characterize the scalar curvature in the limit space.

\section{\bf General Class of Examples}\label{Sec-Class}
\begin{defn}\label{General-Case}
For $i,\ j=1,\ 2,...$, let  $f_{i,j}: \Sph^2\to \R^+$ be positive smooth functions, satisfying the following two requirements
\begin{enumerate}[(i)]
    \item For some $0<T<\infty$
\be
\int_{\Sph^2} f_{i,j} d\vol_{\Sph^2} \leq T,\text{ for all }i,\ j
\ee
where $d\vol_{\Sph^2}$ is the volume form on the standard $\Sph^2$;  
\item and, for all $i,\ j$,
\be
2\leq f_{i,j}\text{ and }\Delta f_{i,j}\leq f_{i,j},
\ee
where $\Delta$ is the Laplacian in the standard sphere.
\end{enumerate}

Let $\{ A_i\}_{i=1}^\infty$  be a sequence of numbers such that $A_i\geq 0$, $A_1>0$ and that $\sum_{i=1}^\infty A_i=K <\infty$. For $j=1,\ 2,\ 3,...$, let $h_j:\Sph^2\to \R^+$ be
\be
h_j(x)=\sum_{i=1}^j A_i f_{i,j}(x).
\ee 
Define a sequence of warped product metrics on $\Sph^2\times \Sph^1$ by
\be
g_j= g_{\Sph^2}+h_j^2 g_{\Sph^1}.
\ee
We use the notation $\Sph^2\times_{h_j}\Sph^1$ to denote $(\Sph^2\times\Sph^1, g_j)$.
\end{defn}

We state in the following theorem that the general class of examples given in Definition \ref{General-Case} satisfies several uniform geometric bounds.
\bt\label{MT}
For the sequence $\Sph^2\times_{h_j}\Sph^1$ defined in Definition \ref{General-Case}, let $\scal_j$ denote the scalar curvature of each $\Sph^2\times_{h_j}\Sph^1$, then we have
\be
\scal_j\geq 0 \text{ for all }j.
\ee
Moreover, there exists $v>0$, $V>0$, and $D>0$ such that
\be
v\leq \vol(\Sph^2\times_{h_j}\Sph^1)\leq  V \text{ and }\diam(\Sph^2\times_{h_j}\Sph^1) \leq D \text{ for all }j.
\ee
\et
We prove this theorem in Section \ref{Sec-Thm}. We prove that each Riemannian manifold is non-degenerate in the following lemma.


\begin{lem}\label{Lem-Tensor-Nondegenerate}
Each Riemannian metric tensor $g_j$ in Definition \ref{General-Case} is strictly positive definite. The sequence of function $\{h_j\}_{j=1}^\infty$ as in Definition \ref{General-Case} is bounded below by $2A_1$, and the sequence of Riemannian metric tensor $\{g_j\}_{j=1}^\infty$ as in Definition \ref{General-Case} satisfies
\be
 g_{j}\geq g_{\Sph^2}+(2A_1)^2 g_{\Sph^1} \quad \text{for all } j
\ee
\end{lem}

\begin{proof}
In Definition \ref{General-Case}, each $f_{i,j}$ is a positive function. Moreover, since we choose $A_i \geq 0$ for all $i$ and $A_1>0$, each $g_j$ is strictly positive definite.

By the choice of the sequence $\{A_i\}_{i=1}^\infty$ in Definition \ref{General-Case} we have
\be
h_j=\sum_{i=1}^j A_i f_{i,j}\geq A_1 f_{1,j}.
\ee
Since in Definition \ref{General-Case} we assume $f_{i,j}\geq 2$ for all $i,\ j$, we have
\be
h_j \geq  2A_1
\ee
for all $j$. The inequality for $g_j$ follows from this.
\end{proof}

In the following remark, we explain how the results from \cite{Tian-Wang} can be applied to the general class of examples in Definition \ref{General-Case}.
\br\label{rmk-mina}
In the general class of example Definition \ref{General-Case}, we do not have the $\mina$ lower bound as required in Theorem 1.3, Theorem  1.7, and Theorem 1.8 of \cite{Tian-Wang}. But in \cite{Tian-Wang}, the $\mina$ lower bound is only used in Proposition 4.10 and part (iii) of Theorem 1.3 to prove that the sequence of warping functions does not converge to the zero function. The $\mina$ lower bound is included in the statement of Theorem 1.7 and 1.8 from \cite{Tian-Wang} to ensure that the limit warping function is not the zero function and hence the limit metric is a (weak) warped product Riemannian metric as defined in Definition 5.1 of \cite{Tian-Wang}.

By Lemma \ref{Lem-Tensor-Nondegenerate}, we know that each $h_j$ is strictly positive and that the sequence is uniformly bounded below by $2 A_1>0$. Hence, $h_j$ cannot converge to the zero function as $j\to\infty$. As a result, by Theorem \ref{MT}, we can still apply Theorem 1.3, Theorem  1.7, and Theorem 1.8 of \cite{Tian-Wang} to the general class of examples in Definition \ref{General-Case}.
\er

By Theorem \ref{MT} and Remark \ref{rmk-mina}, we can apply \cite{Tian-Wang} to the general class of examples in Definition \ref{General-Case}. We have the following corollary:

\begin{cor}\label{cor-conv}
Let $\{\Sph^2\times_{h_j}\Sph^1\}_{j=1}^\infty$ be a sequence of warped product Riemannian manifold as in Definition \ref{General-Case}. Then, there exists a subsequence $g_{j_k}$ and a (weak) warped product Riemannian metric $g_\infty$ such that $g_\infty \in W^{1,p}(\Sph^2\times \Sph^1,g_0)$ for $p\in [1,2)$, and that  $g_{j_k}\to g_\infty$ in $L^q(\Sph^2\times \Sph^1,g_0)$  for all $q\in [1,\infty)$ as $k\to \infty$. 

Moreover, the limit metric $g_\infty$ has non-negative distributional scalar curvature as defined in \cite{Lee-LeFloch}. The total curvature of the sequence $g_{j_k}$ converges to the total scalar curvature of $g_\infty$ as $k\to\infty$.

\end{cor}

\begin{proof}
Combine Theorem \ref{MT} with Theorem 1.7 and 1.8 from \cite{Tian-Wang}. The spaces $L^q(\Sph^2\times \Sph^1,g_0)$ and $W^{1,p}(\Sph^2\times \Sph^1,g_0)$ are defined in Definition 5.3 from \cite{Tian-Wang}.
\end{proof}

\section{\bf Examples of sequences of warped product $\Sph^2\times \Sph^1$}\label{Sec-Examples}
We prove Theorem \ref{MT-New} in this section. 

Firstly, we prove that the class of example in Definition \ref{Sub-General-Case} is included in Definition \ref{General-Case}. We need the following lemmas:
\bl\label{Lem-Single-Vol}
Let $\bar{K}>2$. Let $a>0$, $b\in [2,\bar{K}]$. Let $x_0\in \Sph^2$ such that for $x\in \Sph^2$, $r_0(x)=d_{\Sph^2}(x,x_0)$. Let $f_{a,b}$ be as in Definition \ref{Defn-functions}. Then the function $f_{a,b}(r_0(x))$ satisfies
\be
0<f_{a,b}(r_0(x))\leq f_{0,b}(r_0(x))=-2\ln\left(\sin r_0(x)\right)+b,
\ee
and hence
\begin{align}
    0 &<\int_{\Sph^2} f_{a,b}d\vol_{\Sph^2} \\
    &\leq \int_{\Sph^2} f_{0,b}d\vol_{\Sph^2} \\
    &=8\pi-4\pi\ln 4+4\pi b\\
    &\leq 8\pi-4\pi\ln 4+4\pi \bar{K}.
\end{align}
\el

\begin{proof}
Since $f_{a,b}(r_0(x))$ is a radial function that only depends on $r_0(x)$, with slight abuse of notation, we use $f_{a,b}(r_0)$ and $f_{0,b}(r_0)$ to denote $f_{a,b}(r_0(x))$ and $f_{0,b}(r_0(x))$.

Since $\frac{1+a}{\sin^2 r+a}=1+\frac{1-\sin^2 r}{\sin^2 r+a}$ is a decreasing function of $a$. We have
\be
\ln\left(\frac{1+a}{\sin^2 r+a}\right)\leq -2\ln\left(\sin r\right).
\ee
Hence we have
\be
0<f_{a,b}(r_0)\leq f_{0,b}(r_0).
\ee

In the polar coordinate centered at the point $x_0$ the metric in $\Sph^2$ is
\be
g_{\Sph^2}= dr_0^2+ \sin^2 r_0 d\theta^2.
\ee
Hence $d\vol_{\Sph^2}= \sin r_0 d r_0 d\theta$.
We integrate $f_{0,b}$ over the standard $\Sph^2$ to get
\begin{align}
    \int_{\Sph^2} f_{0,b}  d\vol_{\Sph^2} &=\int_{r_0=0}^{\pi} \int_{\theta=0}^{2\pi} (-2\ln(\sin r_0)+b )\sin r_0 d\theta dr_0\\
    &=4\pi b+2\pi\int_{r_0=0}^{\pi} -2\sin r_0 \ln(\sin r_0)dr_0.
\end{align}
Since the function $-2\ln(\sin r_0)$ has anti-derivative
\be
\int -2\sin r_0\ln(\sin r_0) dr_0 =-2\left(\cos r_0+\ln\left(\tan\frac{r_0}{2}\right)-\cos r_0 \ln(\sin r_0)\right),
\ee
plug in the bounds, then we get 
\be
\int_{\Sph^2} f_{0,b}d\vol_{\Sph^2}     =8\pi-4\pi\ln 4+4\pi b     \leq 8\pi-4\pi\ln 4+4\pi \bar{K}.
\ee
This finishes the proof.
\end{proof}

\bl[Proposition 2.5 from \cite{STW}]\label{Lem-Single-Scal}
Let $a>0$, $b>2$. Let $x_0\in \Sph^2$ such that for $x\in \Sph^2$, $r_0(x)=d_{\Sph^2}(x,x_0)$. Let $f_{a,b}$ be as in Definition \ref{Defn-functions}. Then the function $f_{a,b}(r_0(x))$ satisfies
\be
\Delta f_{a,b}(r_0(x)) \leq f_{a,b}(r_0(x)),
\ee
where $\Delta$ is the Laplacian on the standard $\Sph^2$
\el
\begin{proof}
Refer to Proposition 2.5 from \cite{STW}.
\end{proof}

Now, we prove that the sequence in Definition \ref{Sub-General-Case} is included in Definition \ref{General-Case}.

\begin{prop}\label{Prop-Inclusion}
The sequence $\{f_{a_j,b_i}(r_{ij}(x))\}_{i,j=1}^\infty$ in Definition \ref{Sub-General-Case} satisfies the requirements (i) and (ii) in Definition \ref{General-Case} with $T= 8\pi-4\pi\ln 4+4\pi \bar{K}$. As a result, the sequence in Definition \ref{Sub-General-Case} is a special case of Definition \ref{General-Case}.
\end{prop}
\begin{proof}
Apply Lemma \ref{Lem-Single-Vol} and \ref{Lem-Single-Scal} to each $f_{a_j,b_i}(r_{ij}(x))$, where $x_0=x_{ij}$, $a=a_j$, and $b=b_i$. The inequality $f_{i,j}\geq 2$ follows from the choice that $b_i\geq 2$ for all $i$.
\end{proof}


Now, we prove Theorem \ref{MT-New}.
\begin{proof}[Proof of Theorem \ref{MT-New}]
By Proposition \ref{Prop-Inclusion}, the class of example in Definition \ref{Sub-General-Case} is included in the larger class in Definition \ref{General-Case}. The scalar curvature non-negative condition and the volume and diameter upper bounds are proved in Theorem \ref{MT}. The $L^q$ convergence and $W^{1,q}$ regularity are proved in Corollary \ref{cor-conv}.
\end{proof}

\section{\bf Limit Metric Tensor}\label{Sec-Lim}
The sequences in Example \ref{Example-Case1} and \ref{Example-Case2} have a limit metric tensor. Here, we directly prove the convergence to the limit metric tensor by generalizing the results from \cite{STW}. The $W^{1,p}$ convergence we obtain in Proposition \ref{Prop-Limit} is stronger than the convergence we obtain in \cite{Tian-Wang}. For Example \ref{Example-Case3}, we can still apply the results from \cite{Tian-Wang} to get subsequential convergence as discussed in Corollary \ref{cor-conv}. 

We start with the following lemmas.

\bl[Proposition 3.7 from \cite{STW}]\label{Lem-Single-Gradient-Convergence}
 Let $x_0\in \Sph^2$ such that for $x\in \Sph^2$, $r_0(x)=d_{\Sph^2}(x,x_0)$. Let $f_{a,b}$ be as in Definition \ref{Defn-functions}. Let $\{a_j\}_{j=1}^\infty$ be the sequence chosen in Definition \ref{Sub-General-Case}. Then we have $f_{0,b_i}(r_0(x))\in W^{1,p}(\Sph^2)$ for $p\in [1,2)$, $f_{0,b_i}(r_0(x))\notin W^{1,2}(\Sph^2)$, 
 \be
  f_{a_j,b_i}(r_0(x))\to f_{0,b_i}(r_0(x))
 \ee
in $L^q(\Sph^2)$ for all $q\in [1,\infty)$, and that
 \be
\nabla f_{a_j,b_i}(r_0(x))\to \nabla  f_{0,b_i}(r_0(x))
 \ee
 in $L^p(\Sph^2)$ for all $p\in [1,2)$. Here $\nabla$ is the connection in the standard $\Sph^2$.
\el
\begin{proof}
Refer to Proposition 3.7 from \cite{STW}.
\end{proof}

\bl\label{Lem-Gradient}
Let $f_{a,b}$ be as in Definition \ref{Defn-functions}. Let  $\{A_i\}_{i=1}^\infty,\ \{a_j\}_{j=1}^\infty, \text{ and }\ \{b_i\}_{i=1}^\infty$ be the sequences chosen in Definition \ref{Sub-General-Case}. Let $\{\bar{x}_i\}_{i=1}^\infty$, $\{\hat{x}_i\}_{i=1}^\infty$ be the sequences of points in $\Sph^2$ chosen in Example \ref{Example-Case1} and \ref{Example-Case2} respectively. Then $\sum_{i=1}^\infty A_i f_{0, b_i}(\bar{r}_i(x)),\ \sum_{i=1}^\infty A_i f_{0, b_i}(\hat{r}_i(x))\in W^{1,p}(\Sph^2)$ for $p\in [1,2)$ and 
\begin{align}
\nabla \sum_{i=1}^\infty A_i f_{0, b_i}(\bar{r}_i(x)) &= \sum_{i=1}^\infty A_i  \nabla f_{0, b_i}(\bar{r}_i(x)),\\
\nabla \sum_{i=1}^\infty A_i f_{0, b_i}(\hat{r}_i(x)) &= \sum_{i=1}^\infty A_i  \nabla f_{0, b_i}(\hat{r}_i(x)).
\end{align}
\el
\begin{proof}
We only prove it for Example \ref{Example-Case2}. The sequence in Example \ref{Example-Case1} can be considered a special case of Example \ref{Example-Case2}.

For any smooth vector field $X$ in $\Sph^2$, for all $j$, by integration by parts, we have
\be
\int_{\Sph^2} \left\langle X, \sum_{i=1}^j A_i \nabla f_{0,b_i}(\hat{r}_i(x))\right\rangle d\vol_{\Sph^2}
=-\int_{\Sph^2} \text{div} X \sum_{i=1}^j A_i f_{0,b_i}(\hat{r}_i(x)).
\ee
Here $\langle \cdot,  \cdot \rangle$ and $\text{div}$ are defined using the standard metric in $\Sph^2$.

Since $A_i\geq 0$ and $f_{0,b_i}>0$, by the monotone convergence theorem
\be
\int_{\Sph^2}  \sum_{i=1}^j A_i f_{0,b_i}(\hat{r}_i(x))\xrightarrow{j\to \infty} \int_{\Sph^2}  \sum_{i=1}^\infty A_i f_{0,b_i}(\hat{r}_i(x)),
\ee
By Definition \ref{Sub-General-Case} and Lemma \ref{Lem-Single-Vol} we have
\be
\int_{\Sph^2}\sum_{i=1}^\infty A_i f_{0,b_i}(\hat{r}_i(x))<\infty.
\ee
Therefore, by the H\"older's inequality and the dominated convergence theorem, we have
\be
\int_{\Sph^2}  \text{div} X \sum_{i=1}^j A_i f_{0,b_i}(\hat{r}_i(x))\xrightarrow{j\to \infty} 
\int_{\Sph^2}  \text{div} X  \sum_{i=1}^\infty A_i f_{0,b_i}(\hat{r}_i(x)).
\ee
On the other hand, by the Cauchy-Schwarz inequality, triangle inequality, Lemma \ref{Lem-Single-Gradient-Convergence}, and the dominated convergence theorem we have
\be
\int_{\Sph^2} \left\langle X, \sum_{i=1}^j A_i \nabla f_{0,b_i}(\hat{r}_i(x))\right\rangle d\vol_{\Sph^2} \xrightarrow{j\to \infty}     
 \int_{\Sph^2} \left\langle X, \sum_{i=1}^\infty A_i \nabla f_{0,b_i}(\hat{r}_i(x))\right\rangle d\vol_{\Sph^2}.
\ee
As a result, we have
\be
 \int_{\Sph^2} \left\langle X, \sum_{i=1}^\infty A_i \nabla f_{0,b_i}(\hat{r}_i(x))\right\rangle d\vol_{\Sph^2} =
 -\int_{\Sph^2}  \text{div} X  \sum_{i=1}^\infty A_i f_{0,b_i}(\hat{r}_i(x)),
\ee
for all smooth vector field $X$ in $\Sph^2$, and hence
\be
\nabla \sum_{i=1}^\infty A_i f_{0, b_i}(\hat{r}_i(x)) = \sum_{i=1}^\infty A_i  \nabla f_{0, b_i}(\hat{r}_i(x)).
\ee

By the vector Jensen's inequality, we have for $p\in [1,2)$
\be
\left| \frac{\sum_{i=1}^\infty A_i  \nabla f_{0, b_i}(\hat{r}_i(x))}{K} \right|^p \leq \frac{\sum_{i=1}^\infty A_i \left| \nabla f_{0, b_i}(\hat{r}_i(x)) \right| ^p}{K},
\ee
where $\sum_{i=1}^\infty A_{i}  =K <\infty$. By the monotone convergence theorem and Lemma \ref{Lem-Single-Gradient-Convergence},  for all $p\in[1,2)$ we have
\be
\sum_{i=1}^\infty A_i \left| \nabla f_{0, b_i}(\hat{r}_i(x)) \right| ^p\in L^1(\Sph^2).
\ee
As a result, 
\be
 \sum_{i=1}^\infty A_i f_{0, b_i}(\hat{r}_i(x))\in W^{1,p}(\Sph^2)\text{ for }p\in [1,2).
\ee
\end{proof}

\bl\label{Lem-Convergence-WarpingFunction}
Let $f_{a,b}$ be as in Definition \ref{Defn-functions}. Let  $\{A_i\}_{i=1}^\infty,\ \{a_j\}_{j=1}^\infty, \text{ and }\ \{b_i\}_{i=1}^\infty$ be the sequences chosen in Definition \ref{Sub-General-Case}. Let $\{\bar{x}_i\}_{i=1}^\infty$, $\{\hat{x}_i\}_{i=1}^\infty$ be the sequences of points in $\Sph^2$ chosen in Example \ref{Example-Case1} and \ref{Example-Case2} respectively. Then 
\begin{align}
\sum_{i=1}^j A_i f_{a_j, b_i}(\bar{r}_i(x)) &\to \sum_{i=1}^\infty A_i f_{0, b_i}(\bar{r}_i(x)),\\
\text{ and } \sum_{i=1}^j A_i f_{a_j, b_i}(\hat{r}_i(x)) & \to \sum_{i=1}^\infty A_i f_{0, b_i}(\hat{r}_i(x))
\end{align}
in $W^{1,p}(\Sph^2)$ norm for $p\in [1,2)$.
\el
\begin{proof}
Again, we only prove it for Example \ref{Example-Case2}. The sequence in Example \ref{Example-Case1} can be considered a special case of Example \ref{Example-Case2}.
Think of the summation $\sum_{i=1}^\infty $ as integration with counting measure, then by Jensen's inequality, we have for $q\in [1,\infty)$
\begin{align}
&\left|\frac{ \sum_{i=1}^\infty A_i f_{0, b_i}(\hat{r}_i(x))-    \sum_{i=1}^j A_i f_{a_j, b_i}(\hat{r}_i(x))}{K} \right|^q   \\
&=  \left|   \frac{ \sum_{i=1}^j  A_i  \left[ f_{0, b_i}(\hat{r}_i(x))- f_{a_j, b_i}(\hat{r}_i(x))\right]   + \sum_{i=j+1}^\infty  A_i f_{0, b_i}(\hat{r}_i(x))}{K} \right|^q   \\
&\leq  \frac{1}{K} \sum_{i=1}^j  A_i  \left| f_{0, b_i}(\hat{r}_i(x))- f_{a_j, b_i}(\hat{r}_i(x))\right|^q 
+\frac{1}{K}  \sum_{i=j+1}^\infty  A_i \left| f_{0, b_i}(\hat{r}_i(x)) \right|^q,
\end{align}
where $\sum_{i=1}^\infty A_{i}  =K <\infty$. By Lemma \ref{Lem-Single-Gradient-Convergence} and the summability of the sequence $\{A_i\}_{i=1}^\infty$, we have for $q\in [1,\infty)$
\be
\sum_{i=j+1}^\infty A_i \int_{\Sph^2} \left(f_{0, b_i}(\hat{r}_i(x))\right)^q d\vol_{\Sph^2}\to 0 \text{ as }j\to \infty.
\ee
By Lemma \ref{Lem-Single-Gradient-Convergence}, we have for all $i$
\be
\int_{\Sph^2}\left(f_{0, b_i}(\hat{r}_i(x))- f_{a_j, b_i}(\hat{r}_i(x)) \right)^q d\vol_{\Sph^2} \to 0\text{ as }j\to \infty.
\ee
Note that the convergence here does not depend on $i$. As a result, 
\be
\sum_{i=1}^j A_i f_{a_j, b_i}(\hat{r}_i(x))\to\sum_{i=1}^\infty A_i f_{0, b_i}(\hat{r}_i(x))
\ee
in $L^q(\Sph^2)$ as $j\to\infty$ for all $q\in [1,\infty)$.

By Lemma \ref{Lem-Gradient}, we have $\sum_{i=1}^\infty A_i f_{0, b_i}(\hat{r}_i(x))\in W^{1,p}(\Sph^2)$ for $p\in [1,2)$ and that
\be
\nabla \sum_{i=1}^\infty A_i f_{0, b_i}(\hat{r}_i(x)) = \sum_{i=1}^\infty A_i  \nabla f_{0, b_i}(\hat{r}_i(x)).
\ee
By Jensen's inequality, we have for $p\in [1,2)$
\begin{align}
&\left| \frac{\nabla \sum_{i=1}^\infty A_i f_{0, b_i}(\hat{r}_i(x))-    \nabla\sum_{i=1}^j A_i f_{a_j, b_i}(\hat{r}_i(x))}{K} \right|^p   \\
&\leq \frac{1}{K}   \sum_{i=1}^j A_i  \left|  \nabla f_{0, b_i}(\hat{r}_i(x)) -  \nabla f_{a_j, b_i}(\hat{r}_i(x))\right|^p +\frac{1}{K}  \sum_{i=j+1}^\infty A_i \left|\nabla f_{0, b_i}(\hat{r}_i(x)) \right|^p
\end{align}
By Lemma \ref{Lem-Single-Gradient-Convergence} and the summability of the sequence $\{A_i\}_{i=1}^\infty$, we have for $p\in [1,2)$
\be
\sum_{i=j+1}^\infty A_i \left|\nabla f_{0, b_i}(\hat{r}_i(x)) \right|^p d\vol_{\Sph^2}\to 0 \text{ as }j\to \infty.
\ee
By Lemma \ref{Lem-Single-Gradient-Convergence}, we have for all $i$
\be
\int_{\Sph^2}  \left|  \nabla f_{0, b_i}(\hat{r}_i(x)) -  \nabla f_{a_j, b_i}(\hat{r}_i(x))\right|^p d\vol_{\Sph^2} \to 0\text{ as }j\to \infty.
\ee
Note that the convergence here does not depend on $i$. As a result, 
\be
\sum_{i=1}^j A_i   \nabla f_{a_j, b_i}(\hat{r}_i(x))\to\sum_{i=1}^\infty A_i   \nabla f_{0, b_i}(\hat{r}_i(x))
\ee
in $L^p(\Sph^2)$ as $j\to\infty$ for all $p\in [1,2)$.
\end{proof}

Now we identify the limit tensor and prove the following:
\bp[Proposition 3.9 from \cite{STW}]\label{Prop-Limit}
The sequence of Riemannian metric tensor in Example \ref{Example-Case1} converges in $W^{1,p}(\Sph^2\times \Sph^1,g_0)$ norm to the following metric
\be
\bar{g}_\infty=g_{\Sph^2}+\left(\sum_{i=1}^\infty A_i f_{0, b_i}(\bar{r}_i(x))\right)g_{\Sph^1},
\ee
where $\bar{r}_i(x)=d_{\Sph^2}(x,\bar{x}_i)$, and the sequence $\{\bar{x}_i\}_{i=1}^\infty$ is as chosen in Example \ref{Example-Case1}.
The sequence of Riemannian metric tensor in Example \ref{Example-Case2} converges in $W^{1,p}(\Sph^2\times \Sph^1,g_0)$ norm to the following metric
\be
\hat{g}_\infty=g_{\Sph^2}+\left(\sum_{i=1}^\infty A_i f_{0, b_i}(\hat{r}_i(x))\right)g_{\Sph^1},
\ee
where $\hat{r}_i(x)=d_{\Sph^2}(x,\hat{x}_i)$, and the sequence $\{\hat{x}_i\}_{i=1}^\infty$ is as chosen in Example \ref{Example-Case2}.
The notion $W^{1,p}(\Sph^2\times \Sph^1,g_0)$ is given in Definition 3.5 from \cite{STW}
\ep
\begin{proof}
Combine Lemma \ref{Lem-Convergence-WarpingFunction} with Proposition 3.9 from \cite{STW}.
\end{proof}

\section{\bf Scalar Curvature, Volume, and Diameter}\label{Sec-Thm}
We prove Theorem \ref{MT} in this section. We separate the proof into several different parts.

First, we prove that each manifold in Definition \ref{General-Case} has non-negative scalar curvature:
\begin{prop}[Scalar Curvature]\label{Prop-Scal}
For each $j$, the scalar curvature of $\Sph^2\times_{h_j}\Sph^1$ in Definition \ref{General-Case} is non-negative.
\end{prop}

\begin{proof}
Denote the scalar curvature of $\Sph^2\times_{h_j}\Sph^1$ as $\Scal_j$. By \cite{KK-compact} we have
\be
\Scal_j= 2-2\frac{\Delta h_j}{h_j}.
\ee
Here $\Delta$ is the Laplacian on the standard $\Sph^2$. Since $h_j>0$ for each $j$, the scalar curvature non-negative condition is equivalent to
\be
\Delta  h_j\leq h_j.
\ee
By part (ii) of Definition \ref{General-Case} for each $i,\ j$
\be
\Delta f_{i,j}\leq  f_{i,j}.
\ee
Moreover, since we chose $A_i \geq 0$ for each $i$, we have
\be
\Delta h_j=\sum_{i =1}^j A_i \Delta   f_{i,j}\leq \sum_{i =1}^j A_i f_{i,j}= h_j.
\ee
This finishes the proof.
\end{proof}

Before we prove the uniform volume upper bound, we prove the following lemma:
\bl\label{Lem-Vol}
For each $j$, for $\Sph^2\times_{h_j}\Sph^1$ in Definition \ref{General-Case}
we have
\be
\int_{\Sph^2} h_j d\vol_{\Sph^2} \leq K T,
\ee
where $T$ is the constant in Definition \ref{General-Case}. 
\el
\begin{proof}
By definition of $h_j$ we have
\be
 \int_{\Sph^2} h_j d\vol_{\Sph^2}=\sum_{i=1}^j A_i \int_{\Sph^2} f_{i,j} d\vol_{\Sph^2}.
\ee
By part (i) of Definition \ref{General-Case}, for each $i$ we have
 \be
 \int_{\Sph^2} f_{i,j} d\vol_{\Sph^2} \leq T.
 \ee
 Hence, 
 \be
 \int_{\Sph^2} h_j d\vol_{\Sph^2} \leq T\sum_{i=1}^j A_i \leq KT.
\ee
This finishes the proof.
\end{proof}

We prove both the volume upper bound and the volume lower bound in the following proposition:
\begin{prop}[Volume]\label{Prop-Vol}
For each $j$, the volume of $\Sph^2\times_{h_j}\Sph^1$ in Definition \ref{General-Case} is bounded below by $16\pi^2 A_1$ and is bounded above by $2\pi KT$.
\end{prop}
\begin{proof}
    Since $g_j=g_{\Sph^2}+h_j^2 g_{\Sph^1}$, in the coordinate $\{r,\theta, \varphi\}$, we have
    \be
   g_j= \begin{pmatrix}
1 &0 &0\\
0& \sin^2 r &0\\
0 &0& h_j^2.
    \end{pmatrix}
    \ee
    Hence, the volume form is
    \be
    d\vol_j= h_j d\vol_{\Sph^2}d\vol_{\Sph^1}.
    \ee
Denote the volume of $\Sph^2\times_{h_j}\Sph^1$ as $\vol_j$ then we have
\begin{align}
\vol_j &=\int_{\Sph^2}\int_{\Sph^1} h_j d\vol_{\Sph^1}d\vol_{\Sph^2} \\
&=2\pi \int_{\Sph^2} h_j d\vol_{\Sph^2}.
\end{align}
Hence by Lemma \ref{Lem-Tensor-Nondegenerate}, we have
\be
\vol_j \geq 16\pi^2,
\ee
and by Lemma \ref{Lem-Vol}, we have
\be
\vol_j \leq 2\pi K T.
\ee
This finishes the proof.
\end{proof}

Next, we prove that the diameter in each one of the manifolds in Definition \ref{General-Case} is uniformly bounded above. We need the following lemma:
\bl[Lemma 2.12 from \cite{STW}]\label{Lem-Single-Diam}
Let $h:\Sph^2\to \R$ be a smooth, positive function. Given any $p_1=(r_1,\theta_1,\varphi_1)$ and $p_2=(r_2,\theta_2,\varphi_2)$ in the warped product space $\Sph^2\times_{h}\Sph^1=(\Sph^2\times \Sph^1, g_{\Sph^2}+h^2 g_{\Sph^1})$ with the distance function $d_h$. Then we have the inequality
\be
d_h((r_1,\theta_1,\varphi_1),(r_2,\theta_2,\varphi_2))\leq |r_1-r_2|+\sin(r_2) d_{\Sph^1}(\theta_1,\theta_2)+h(r_2,\theta_2) d_{\Sph^1}(\varphi_1,\varphi_2).
\ee
\el
\begin{proof}
Refer to Lemma 2.12 from \cite{STW}.
\end{proof}

\begin{prop}[Diameter]\label{Prop-Diam}
For each $j$, the diameter of $\Sph^2\times_{h_j}\Sph^1$ in Definition \ref{General-Case} is bounded above by 
\be
4\pi +2\pi KT,
\ee
where $K$ is as in Definition \ref{General-Case} and $T$ is as in Definition \ref{General-Case}.
\end{prop}
\begin{proof}
The proof is similar to Proposition 2.13 in \cite{STW}. By Lemma \ref{Lem-Vol}, for each $j$ there exists $y_j\in\Sph^2$ such that
\be
h_j(y_j)\leq K T.
\ee
Suppose in a fixed polar coordinate $\{r,\ \theta\}$ in $\Sph^2$ we have 
\be
y_j=(r_j,\theta_j).
\ee
In $\Sph^2\times_{h_j} \Sph^1$, in the coordinate $\{r,\ \theta\ \varphi\}$, define
\be
p=(r_j,\ \theta_j, \pi).
\ee
Here $\varphi=\pi$ is chosen randomly. 

Denote the distance function in $\Sph^2\times_{h_j}\Sph^1$ as $d_j$. For any $q\in \Sph^2\times_{h_j} \Sph^1$, by Lemma \ref{Lem-Single-Diam}, we have
\be\label{Eqn-djpq}
d_j(p,q)\leq 2\pi+\pi  KT,
\ee
because $|r(p)-r(q)|\leq \pi$, $\sin (r(p)) d_{\Sph^1}(\theta(p),\theta(q))\leq \pi$ and
\be
h_j(y_j)d_{\Sph^1}(\varphi_1,\varphi_2)\leq \pi  KT.
\ee
For any $q_1,\ q_2\in \Sph^2\times_{h_j}\Sph^1$, by the triangle inequality we have
\be
d_j(q_1,q_2)\leq d_j(q_1,p)+d_j(q_2,p).
\ee
Apply equation (\ref{Eqn-djpq}) twice, then we have
\be
d_j(q_1,q_2)\leq 4\pi +2\pi T.
\ee
\end{proof}

Combine all of the above, then we have a proof of Theorem \ref{MT}:
\begin{proof}[Proof of Theorem \ref{MT}]
The scalar curvature non-negative condition is proved in Proposition \ref{Prop-Scal}. The uniform volume upper bound and lower bound are proved in Proposition \ref{Prop-Vol} with $v=16\pi^2 A_1$ and $V=2\pi KT$, and the uniform diameter upper bound is proved in Proposition \ref{Prop-Diam} with $D=4\pi+2\pi KT$.
\end{proof}

\section{\bf Scalar Curvature in the Limit Space}\label{Sec-Scalar-Curvature}

In this section, we ask the question of how to characterize the scalar curvature in the limit space. As shown in Corollary \ref{cor-conv}, we can use the notion of distributional scalar curvature as defined in \cite{Lee-LeFloch} to characterize the scalar curvature in the limit space as a non-negative distribution.



There are several other possible ways to characterize scalar curvature in the limit space. For example, in \cite{B-G-GAFA}, Burkhardt-Guim defined the pointwise scalar curvature lower bounds for continuous Riemannian metrics. See also \cite{B-G-SIGMA-survey}.
\begin{question}
Can we start Ricci flow from a closed manifold with a $W^{1,p}$ Riemannian metric tensor when $p$ is small? For example, in $\Sph^2\times \Sph^1$, can we start a Ricci flow from a warped product Riemannian metric tensor that is $W^{1,p}$ for $p\in [1,2)$? If it is possible, can we follow \cite{B-G-GAFA} to define a scalar curvature lower bound for $W^{1,p}$ Riemannian metric tensor using Ricci flow? 
\end{question}
Here, one challenge is the low regularity for sequences in Definition \ref{General-Case}. By \cite{STW} and \cite{Tian-Wang}, we only have $W^{1,p}$ regularity for $1\leq p<2$ for the limit space. Another challenge is the size of the singular sets. The author is working with Christina Sormani to prove that for the extreme limit example in \cite{STW}, after metric completion, the singular set in the limit space has Hausdorff dimension $1$. As demonstrated in Example \ref{Example-Case2}, we may have dense singular sets with Hausdorff dimension $1$. The characterization of the scalar curvature lower bound using the notion of Ricci flow may be challenging when dealing with dense singular sets.

Another way to characterize the scalar curvature in the limit space is to use the volume-limit notion of the scalar curvature. We know that for a Riemannian manifold $M$, the scalar curvature at a point $p$ is equal to the limit
\be
\lim_{r\to 0} \left(\frac{\vol_{\R^3}(B(0,r))-\vol_{M}(B(p,r))}{r^2 \vol_{\R^3}(B(0,r))}\right).
\ee
Here, $B(0,r)$ denotes the geodesic ball centered at the origin with radius $r$ in the three-dimensional Euclidean space, whereas $B(p,r)$ denotes the geodesic ball centered at $p$ with radius $r$ in $M$. The notation $\vol_{\R^3}$ refers to the volume in the three-dimensional Euclidean space, while $\vol_M$ refers to the volume in $M$. We can apply this notion to Riemannian manifolds with low regularity or even metric spaces since it only uses the volume of geodesic balls. This notion is used in \cite{BDS-Sewing}, \cite{BS-seq}, and \cite{KX-Drawstring}.


%

We ask the following question
\begin{question}
Can we use the volume limit notion to show that the scalar curvature is non-negative for the limit space of the sequence in Definition \ref{General-Case}? Similarly, can we use the volume limit notion to show that the scalar curvature is non-negative for the limit space of general warped product $\Sph^2\times \Sph^1$ sequences as studied in \cite{Tian-Wang}?
\end{question}
The challenge here is that we don't have a good way to characterize the geodesic ball in the limit space of the sequences in Definition \ref{General-Case}. In \cite{BDS-Sewing}, \cite{BS-seq}, and \cite{KX-Drawstring}, the authors use the fact that the geodesic ball at the pulled point looks like a geodesic ball in lower dimensional manifolds. For details, see Lemma 6.3 of \cite{BDS-Sewing} or Example 4.2 of \cite{BS-seq}. For the sequence in \cite{STW}, we do not have a similar description of geodesic balls of the limit space. In Example \ref{Example-Case1},  Example \ref{Example-Case2}, and Example \ref{Example-Case3}, there are infinitely many poles in the limit space, it can be more challenging to use the volume limit notion to show that the scalar curvature is non-negative.

In \cite{Gromov-Dirac}, Gromov suggested that we can use Riemannian polyhedra to study scalar curvature comparison. Li \cite{Li-poly} proved a comparison theorem for three-dimensional Riemannian polyhedra with non-negative scalar curvature. We ask the following question.
\begin{question}
Can we construct Riemannian polyhedra in the limit space of Example \ref{Example-Case1}, Example \ref{Example-Case2}, and Example \ref{Example-Case3}? If possible, then, can we prove that these Riemannian polyhedra satisfiy the inequalities as in Theorem 1.4 of \cite{Li-poly} when compared with the corresponding polyhedra in the three-dimensional Euclidean space?
\end{question}

Here, the low regularity of the limit space can be difficult to deal with. By \cite{STW} and \cite{Tian-Wang}, we only expect the limit space to have $W^{1,p}$ regularity for $p\in [1,2)$. But Li's work \cite{Li-poly} only works for Riemannian metric tensor that is at least $C^0$. Another challenge is that it is difficult to find the geodesics in the limit space, and hence it is difficult to construct the polyhedra such that each edge is a geodesic.

\newpage

\bibliographystyle{plain}
\bibliography{main.bib}

\end{document}